\documentclass{amsart}
\usepackage{graphicx}
\usepackage{a4wide}
\usepackage{amssymb}

\newtheorem{theorem}{Theorem}[section]    
  
\newtheorem{claim}[theorem]{Claim}

\newtheorem{corollary}[theorem]{Corollary} 

\theoremstyle{definition}

\newtheorem{observation}{Observation} 
\newtheorem{example}[theorem]{Example}    
\numberwithin{equation}{section}

\newcommand{\Q}{\mathbb{Q}}

\newcommand{\e}{\varepsilon}

\title[On the $3$-dimensional invariant for cyclic contact branched coverings]{On the $3$-dimensional invariant for cyclic contact branched coverings}
\author{Tetsuya Ito}
\thanks{The author was partially supported by JSPS KAKENHI, Grant Number 15K17540.}
\address{Research Institute for Mathematical Sciences, Kyoto university, Kyoto, 606-8502, Japan}
\email{tetitoh@kurims.kyoto-u.ac.jp}
\urladdr{http://www.kurims.kyoto-u.ac.jp/~tetitoh/}
\date{\today} 

\begin{document}
\begin{abstract}
We give a formula of $3$-dimensional invariant for a cyclic contact branched covering of the standard contact $S^{3}$. 
\end{abstract}

\maketitle

\section{Introduction}

Let $\widetilde{M} \rightarrow M$ be a branched covering of a 3-manifold $M$, branched along a link $K \subset M$. When $M$ has a contact structure $\xi$ and $K$ is a transverse link in the contact 3-manifold $(M,\xi)$, $\widetilde{M}$ has the natural contact structure $\widetilde{\xi}$. We call the contact 3-manifold  $(\widetilde{M},\widetilde{\xi})$ the \emph{contact branched covering} of $(M,\xi)$, branched along the transverse link $K$.

Let $(M,\xi)$ be a $p$-fold cyclic contact branched covering of $(S^{3},\xi_{std})$ (the standard contact $S^{3}$), branched along a transverse link $K$. In \cite[Theorem 1.4]{hkp}, it is shown that the euler class $e(\xi)$ is zero, and the 3-dimensional invariant $d_{3}(\xi) \in \Q$ (See \cite{go} for definition) only depends on a topological link type of $K$ and its self-linking number. 

In this note, we show a direct formula of $d_{3}(\xi)$ in terms of its branch locus $K$.

\begin{theorem}
\label{theorem:main}
If a contact 3-manifold $(M,\xi)$ is a $p$-fold cyclic contact branched covering of  $(S^{3},\xi_{std})$, branched along a transverse link $K$, then
\[ d_{3}(\xi) = -\frac{3}{4}\sum_{\omega: \omega^{p}=1}\sigma_{\omega}(K)- \frac{p-1}{2}sl(K)-\frac{1}{2}{p}.\]
Here $\sigma_{\omega}(K)$ denotes the Tristram-Levine signature, the signature of $(1-\omega)A + (1-\overline{\omega})A^{T}$, where $A$ denotes the Seifert matrix for $K$, and $sl(K)$ denotes the self-linking number.
\end{theorem}

Thus, our formula tells us that $d_{3}(\xi)$ actually only depends on the concordance class of $K$ and the self-linking number. Also, by slice Bennequin inequality \cite{ru}, it also shows that the smooth 4-genus $g_{4}(K)$ of $K$ gives a lower bound of $d_{3}(\xi)$.

\begin{corollary}
If a contact 3-manifold $(M,\xi)$ is a $p$-fold cyclic contact branched covering of $(S^{3},\xi_{std})$ branched along $K$, then $d_{3}(\xi) \geq -\frac{5}{2}(p-1)g_{4}(K)-\frac{1}{2}$.
\end{corollary}

\section{Proof}

\begin{proof}[Proof of Theorem \ref{theorem:main}]

Let $(M,\xi)$ be a $p$-fold cyclic contact branched covering, branched along a transverse link $K$ in $(S,\xi_{std})$.
We put the transverse link $K$ as a closed braid, the closure of an $m$-braid $\alpha$.

Let $(S,\psi)$ be the open book decomposition of $(S^{3},\xi_{std})$, whose binding is the $(p,m)$-torus link. Inside $S^{3}$, the page $S$ is an obvious Seifert surface of the $(p,m)$-torus link which we view as the closure of the $p$-braid $(\sigma_{1}\cdots \sigma_{m-1})^{p}$ as we illustrate in Figure \ref{fig:curves}.

Topologically, the page $S$ is the $p$-fold cyclic branched covering of the disk $D^{2}$, branched along $m$-points. Let $\pi:B_{m}=MCG(D^{2}\setminus\{m \text{ points}\}) \rightarrow MCG(S)$ be the map induced by the branched covering map, which is explicitly is written by
$\pi(\sigma_{i}) = D_{i,1}\cdots D_{i,p-1}$ \cite[Lemma 3.1]{hkp}.
Here $D_{i,j}$ denotes the right-handed Dehn twist along the curve $C_{i,j}$ on $S$, given in Figure \ref{fig:curves}. 
(Here we are assuming that $MCG(S)$ acts on $S$ from left, so $D_{i,1}\cdots D_{i,p-1}$ means $D_{i,p-1}$ comes first and $D_{i,1}$ last.)

An important observation is that in $(S^{3},\xi_{std})$, the curves $C_{i,j}$ are realized as the Lergendrian unknot with $tb=-1, rot=0$.

By using $D_{i,j}$, $\psi$ is written by
\[ \psi=\pi(\sigma_{m-1} \cdots\sigma_{2}\sigma_{1})=(D_{m-1,1}\cdots D_{m-1,p-1}) \cdots(D_{2,1}\cdots D_{2,p-1})  (D_{1,1}\cdots D_{1,p-1}). \]
Also, $(S,\phi=\pi(\alpha))$ gives the open book decomposition of $(M,\xi)$.

\begin{figure}[htb]
\begin{center}
\includegraphics*[bb=84 636 357 730,width=80mm]{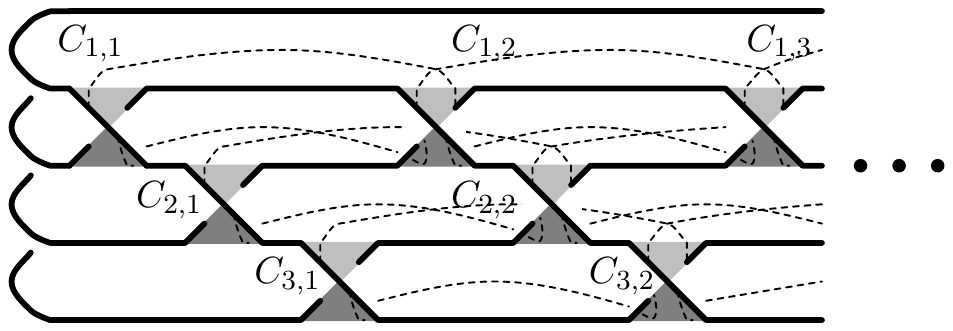}
\caption{Page $S$ of the open book $(S, \psi)$ inside $S^{3}$.}
\label{fig:curves}
\end{center}
\end{figure}

First we draw the surgery diagram of $(M,\xi)$ from its open book decomposition $(S,\phi)$, following the discussion in \cite[Section 3]{hkp}. 
We take a factorization of the braid $(\sigma_{1}^{-1}\cdots \sigma_{m-1}^{-1})\alpha$
\begin{equation}
\label{eqn:factbraid}
(\sigma_{1}^{-1}\cdots \sigma_{m-1}^{-1})\alpha = \sigma_{i_1}^{\e_1} \cdots \sigma_{i_n}^{\e_n} \qquad (\e_j \in \{\pm 1\}, \  i_j \in \{1,\ldots,m-1\})
\end{equation}
 by the standard generators $\{\sigma_1^{\pm 1},\ldots, \sigma_{m-1}^{\pm 1}\}$ of $B_{m}$. By replacing each $\sigma_{i}^{\pm 1}$ in (\ref{eqn:factbraid}) with the sequence of Dehn twists $(D_{i,1}\cdots D_{i,p-1})^{\pm 1}$, we have the factorization of $\psi^{-1}\phi$ by Dehn twists $D_{i,j}^{\pm 1}$, 
\begin{equation}
\label{eqn:factphi}
\psi^{-1}\phi =\prod_{j=1}^{n}(D_{i_j,1}\cdots D_{i_j,p-1})^{\e_j}.
\end{equation}

For each Dehn twist $D_{i,j}^{\pm 1}$ in the factorization (\ref{eqn:factphi}) we put a curve $C_{i,j}$ on distinct pages on the open book $(S,\psi)$, so that it is a Legendrian unknot with $tb=-1$, $rot=0$ in $(S^{3},\xi_{std})$.
Then $(M,\xi)$ is obtained by the contact surgery along the resulting Legendrian link. Here the surgery coefficient of a component is $(-1)$ (resp. $(+1)$) if it comes from a positive (resp. negative) Dehn twist.

The factor $\sigma_{i}$ in the factorization (\ref{eqn:factbraid}) gives a sequence of Dehn twists $(D_{i,1}\cdots D_{i,p-1})$ in the factorization (\ref{eqn:factphi}). The Legendrian curves $C_{i,1},\ldots C_{i,p-1}$, put in different pages (so that $C_{i,p-1}$ comes first and $C_{i,1}$ last), produce the $(p-1)$ component Legendrian link as we draw in Figure \ref{fig:csd_local} (a). 
Similarly, $\sigma_{i}^{-1}$ in the factorization (\ref{eqn:factbraid}) gives a sequence of Dehn twists $D_{i,p-1}^{-1}\cdots D_{i,1}^{-1}$ in the factorization (\ref{eqn:factphi}), which produce the $(p-1)$ component Legendrian unlink as we draw in Figure \ref{fig:csd_local} (b).

\begin{figure}[htb]
\begin{center}
\includegraphics*[bb=76 663 392 741,width=100mm]{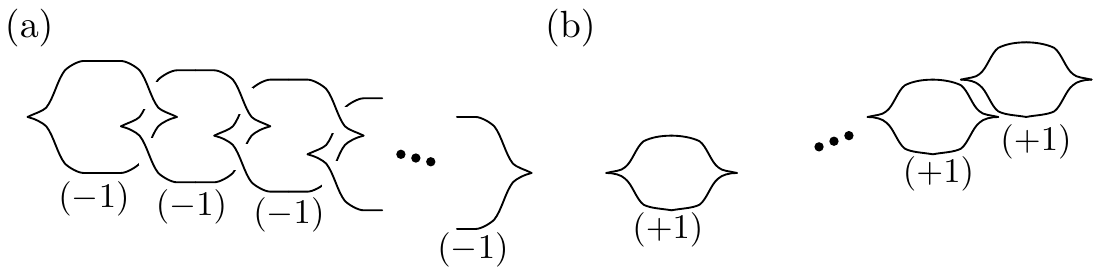}
\caption{The contribution of $\sigma_{i}^{\pm 1}$ in the resulting contact surgery diagram}
\label{fig:csd_local}
\end{center}
\end{figure}

These local contributions of surgery links interact each other, whose linking patterns can be chased by looking the page $S$ in Figure \ref{fig:curves}, as we summarize as follows (cf. \cite[Fig 11, Remark 3.3]{hkp}):

\begin{observation}
\label{obs:cslink}
 Let $\mathcal{L}_{i_k}^{\e_k}= C_{i_k,1}\cup \cdots \cup C_{i_k,p-1}$ and $\mathcal{L}_{i_l}^{\e_l} = C_{i_l,1}\cup \cdots \cup C_{i_l,p-1}$ be the sub Legendrian links in the contact surgery diagram of $M$, that comes from the $k$-th factor $\sigma_{i_k}^{\e_k}$ and $l$-th factor $\sigma_{i_l}^{\e_l}$ in the factorization (\ref{eqn:factbraid}), with $k < l$.

Then the components $C_{i_{k},s}$ and $C_{i_l,t}$ link forms a (topological) positive Hopf link, if and only if $i_{k}\in \{i_{l},i_{l}+1\}$. 
Otherwise, two components $C_{i_{k},s}$ and $C_{i_l,t}$ are disjoint. 
(See Figure \ref{fig:csd_localnest}.)
\end{observation}

\begin{figure}[htb]
\begin{center}
\includegraphics*[bb= 77 609 332 730,width=80mm]{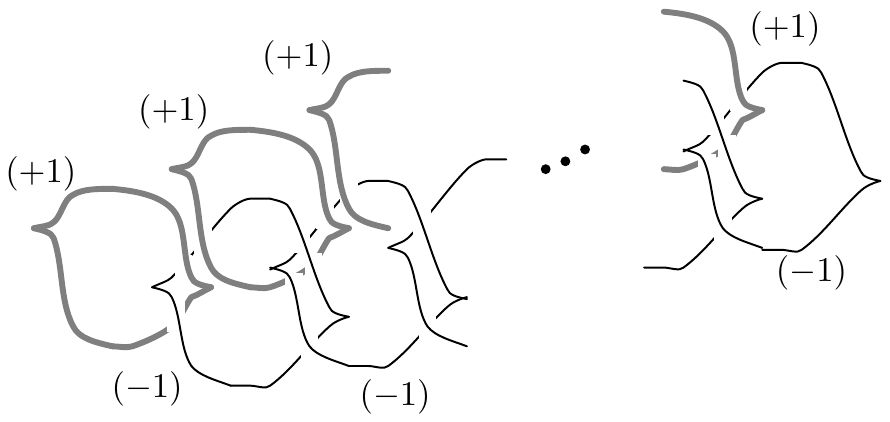}
\caption{How the local contribution of $\sigma_{i}^{\pm 1}$ in the contact surgery diagram links each other. Here we illustrate contributions for $\sigma_{i_k}$ (depicted by black line) and $\sigma_{i_l}^{-1}$ (depicted by gray line) with $k<l$ in the factorization (\ref{eqn:factbraid}), for the case $i_{k} \in \{i_l,i_l+1\}$.}
\label{fig:csd_localnest}
\end{center}
\end{figure}

The contact surgery diagram provides a 4-manifold $X$ that bounds $M$. By \cite[Corollary 3.6]{dgs}, 
\[ d_{3}(\xi) = \frac{1}{4}\left(-3\sigma(X) - 2\chi(X) \right) + q, \]
where $q$ is the number of $(+1)$-contact surgeries, and $\chi(X)$ is the euler characteristic of $X$. Note that the term $c^{2}$ in the formula \cite[Corollary 3.6]{dgs} disappears since each component of the surgery link has zero rotation number. Let $e_{+}$ and $e_{-}$ be the number of positive and negative Dehn twist in the factorization (\ref{eqn:factbraid}). Since each factor $\sigma_{i}^{\pm 1}$ produces $(p-1)$ $(\mp 1)$ contact surgeries along unknots, 
\begin{eqnarray*}
d_{3}(\xi)& = &-\frac{3}{4}\sigma(X) -\frac{1}{2}\bigl( (p-1)e_{+}+(p-1)e_{-}+1 \bigl) + (p-1)e_{-} \\
& = &  -\frac{3}{4}\sigma(X)  -\frac{p-1}{2}(e_{+}-e_{-}) -\frac{1}{2}.
\end{eqnarray*}
By Bennequin's formula $sl(K) =e_{+} - e_{-}+1$, hence
\begin{equation}
\label{eqn:conseq1} d_{3}(\xi) = -\frac{3}{4} \sigma(X) -\frac{p-1}{2}sl(K) -\frac{p}{2}. 
\end{equation}

It remains to compute $\sigma(X)$. 
Take a factorization of the braid $\alpha$ given by
\begin{equation}
\label{eqn:factbraid2}
\alpha = \sigma_{m-1}\cdots \sigma_{1}\sigma_{i_1}^{\e_1} \cdots \sigma_{i_n}^{\e_n} \qquad (\e_j \in \{\pm 1\}, \  i_j \in \{1,\ldots,m-1\}).
\end{equation}

Let $\Sigma \subset S^{3}=\partial B^{4}$ be the canonical Seifert surface of $K$ that comes from the factorization (\ref{eqn:factbraid2}). Namely, $\Sigma$ is made of $m$ disks $\{D_1,\ldots,D_m\}$, with twisted bands connecting $i$-th and $(i+1)$-st disk for each $\sigma_{i}^{\pm 1}$ in the factorization (\ref{eqn:factbraid}).

The following is the most crucial observation in our computation.

\begin{claim}
\label{claim:key}
Let $W$ be the $p$-fold cyclic branched covering of $B^{4}$ branched along $\Sigma$ (pushed into the interiors of $B^{4}$). Then $X$ is diffeomorphic to $W$.
\end{claim}

\begin{proof}[Proof of Claim]
We draw a Kirby diagram of $W$, following \cite[Section 2]{ak} (see also \cite[Section 6.3]{gs}).

Take a handle decomposition of $\Sigma$ so that the $0$-handle is $D_1 \cup h_1 \cup D_2 \cup \cdots \cup h_{m-1} \cup D_{m}$, where $h_i$ is the twisted band coming from the $(m-i)$-th factor $\sigma_{i}$ of the factorization (\ref{eqn:factbraid2}), and that the 1-handles are the rest of twisted bands.
We put $\Sigma$ in the 3-space so that the $0$-handle is the unit disc in the $x$-$y$ plane, and that $1$-handles are contained in the upper half-space.
Then the Kirby diagram of $W$ is obtained by ``symmetrizing'' the cores of $1$-handles of $\Sigma$ whose framings are determined by the framings of the core of $1$-handles. Except the simplest case $p=2$, which we will explictly illustrate later in Example \ref{exam:double}, the diagram is complicated and it is not easy to draw the whole diagram -- however, the contribution of $1$-handle in the resulting Kirby diagram, and how they interact each other is simple. See Figure \ref{fig:branchcov} below.

\begin{figure}[htb]
\begin{center}
\includegraphics*[bb=91 540 337 729,width=70mm]{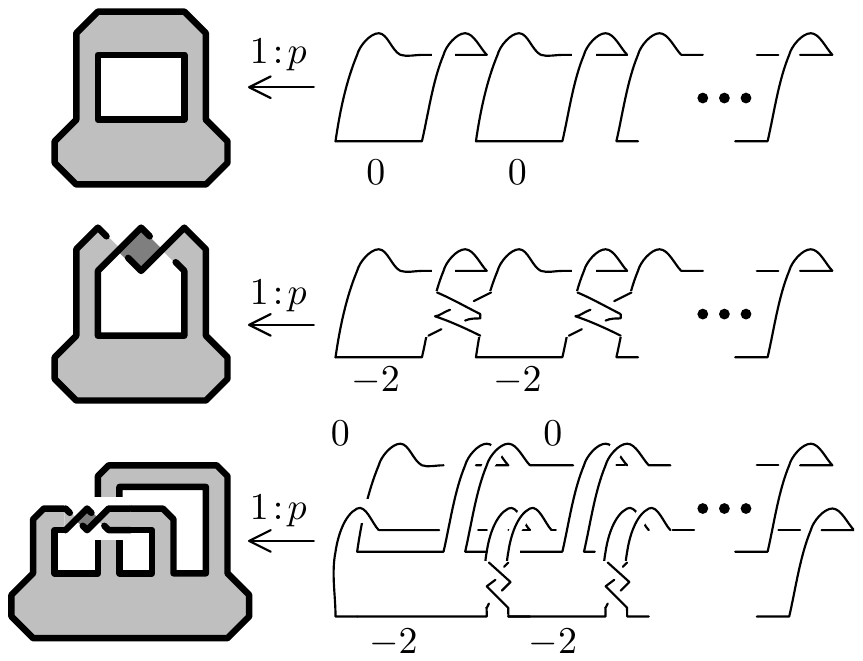}
\caption{Branched covering of Seifert surface: $1$-handle contribution, and how these contributions interact each other (when they are nested).}
\label{fig:branchcov}
\end{center}
\end{figure}
 
To put $\Sigma$ in such a convenient position, first we flip the $1$st disc $D_{1}$, by untwisting the band $h_{1}$ (see Figure \ref{fig:sd1}(a) --(d)). This simplifies the $0$-handle of $\Sigma$, and iterating this flipping procedure, eventually we put $\Sigma$ in such a convenient position (see Figure \ref{fig:Seifert}). In this process, all 1-handles gains negative half twist, so in the final position, the framing of $1$-handle is $(-1)$ if it comes from positive generator, and is $0$ if it comes from negative generator.

\begin{figure}[htb]
\begin{center}
\includegraphics*[bb=84 581 503 744,width=130mm]{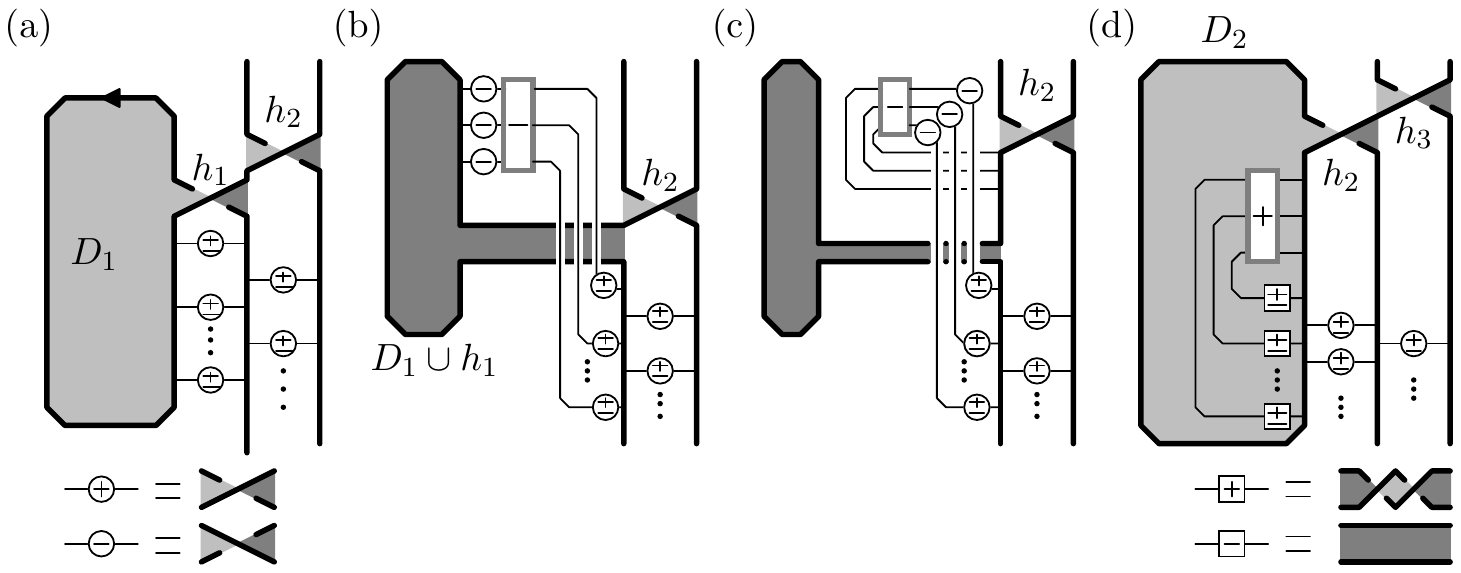}
\caption{Putting the Seifert surface $\Sigma$ into a nice position, by flipping $D_1$ along $h_{1}$. Here we draw 1-handle by a line with $\pm$-sign coming from coresponding generator $\sigma_{i}^{\pm}$. The gray box inside $\pm$ represents positive and negative half twist.}
\label{fig:sd1}
\end{center}
\end{figure}

\begin{figure}[htb]
\begin{center}
\includegraphics*[bb=108 659 292 729, width=75mm]{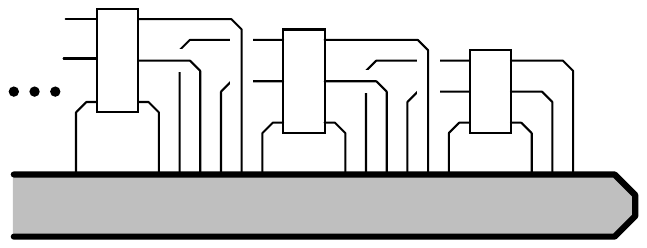}
\caption{The canonical Seifert surface $\Sigma$, put in a convenient position for drawing Kirby diagram. A box represents the positive half twist, and each 1-handle depicted by line has either $0$ or $-1$ framing.}
\label{fig:Seifert}
\end{center}
\end{figure}

From this procedure, we observe:

\begin{observation}
\label{obs:handlenest}
The $1$-handles $h_{k}$ and $h_{l}$ of $\Sigma$, coming from the $k$-th and $l$-th factor $\sigma_{i_k}^{\e_k}$ and $\sigma_{i_l}^{\e_l}$ in the factorization (\ref{eqn:factbraid}) ($k<l$), nest each other in Figure \ref{fig:Seifert} if and only if $i_{k} \in \{i_{l}, i_{l}+1\}$ (see Figure \ref{fig:handlenest})
\end{observation}

\begin{figure}[htb]
\begin{center}
\includegraphics*[bb= 86 658 303 727,width=70mm]{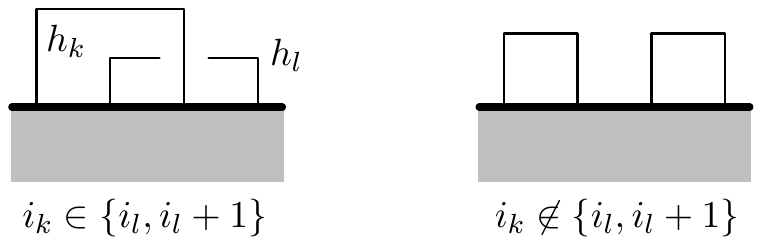}
\caption{How $1$-handles of $\Sigma$ nest each other, when we put $\Sigma$ in a convenient position as in Figure \ref{fig:Seifert}.}
\label{fig:handlenest}
\end{center}
\end{figure}

Recall that each component of the contact surgery diagram of $M$ has $tb=-1$, so $(-1)$ and $(+1)$ contact surgery corresponds to $-2$ and $0$ topological surgery,  respectively. Hence each factor $\sigma_{i}^{\pm 1}$ in (\ref{eqn:factbraid}) contributes the same $(p-1)$ components framed link in the surgery diagram of $X$ and $W$ (compare Figure \ref{fig:csd_local} and Figure \ref{fig:branchcov}). Moreover, from Observation \ref{obs:cslink} and Observation \ref{obs:handlenest}, these local contributions link other part of the diagrams, in exactly the same way (compare  Figure \ref{fig:csd_localnest} and Figure \ref{fig:branchcov}).
Thus, comparisons of the construction of the surgery diagrams for $X$ and $W$ proves that they are completely the same diagram.

\end{proof}

Claim \ref{claim:key}, together with a well-known fact on Tristram-Levine signature (see \cite[Theorem 12.6]{kau}, for example) shows
\begin{equation}
\label{eqn:conseq2}
\sigma(X) = \sigma(W)=\sum_{\omega:\omega^{p}=1}\sigma_{\omega}(K).
\end{equation}
The equalities (\ref{eqn:conseq1}) and $(\ref{eqn:conseq2})$ completes the proof.
\end{proof}

\begin{example}[The case double branched covering]
\label{exam:double}

In the case $p=2$, the contact double branched covering, it is much easier to treat and draw the surgery diagram of $X$ and $W$. Here we give more explicit illustrations of surgery diagrams.

Let $(M,\xi)$ be a contact double branched covering branched along the closure of an $m$-braid $\alpha$. 
We begin with the open book $(S,\psi)$ whose binding is $(m,2)$-torus link.
To visualize its symmetry, we view the the page $S$ as the $(m-1)$-times plumbing of an annulus $A_{i}$ that is the boundary of the positive Hopf link, as illustrated in Figure \ref{fig:plumb}. As an element of the mapping class group of $S$, the standard generator $\sigma_{i}$ lifts to the right-handed Dehn twist along the core of an annulus $A_{i}$.

\begin{figure}[htb]
\begin{center}
\includegraphics*[bb=83 619 275 724,width=55mm]{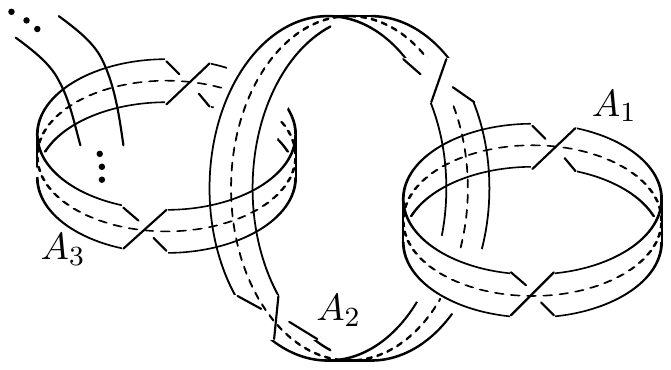}
\caption{Page $S$ of an open book of $(S^{3},\xi_{std}$ whose binding is the $(m,2)$-torus link.}
\label{fig:plumb}
\end{center}
\end{figure}

By taking a factorization of the braid $\alpha$, following the discussion in the proof of Theorem \ref{theorem:main}, we get a contact surgery diagram of $(M,\xi)$,  as we draw in Figure \ref{fig:csdiagram}. On the other hand, the Kirby diagram of $W$ is obtained by ``doubling'' the core of $1$-handles of the canonical Seifert surface $\Sigma$, as we show in Figure \ref{fig:Kirbydiagram}.

\begin{figure}[htbp]
\begin{center}
\includegraphics*[bb=83 633 361 730,width=80mm]{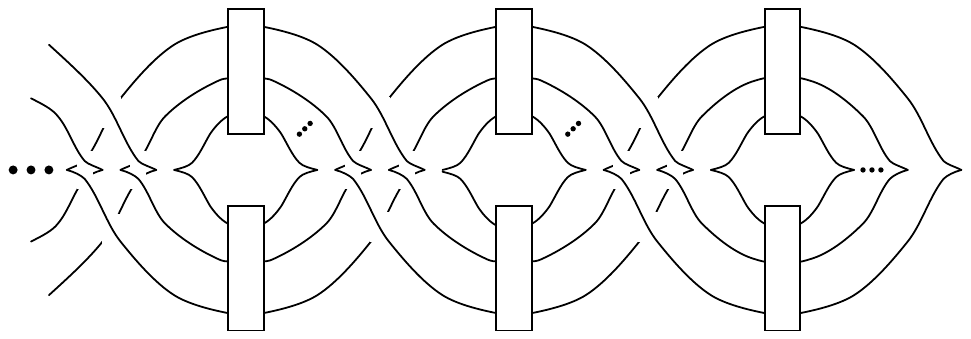}
\caption{A contact surgery diagram of contact double branched covering. A box represents the positive half twist.}
\label{fig:csdiagram}
\end{center}
\end{figure}

\begin{figure}[htb]
\begin{center}
\includegraphics*[bb=58 509 331 727,width=80mm]{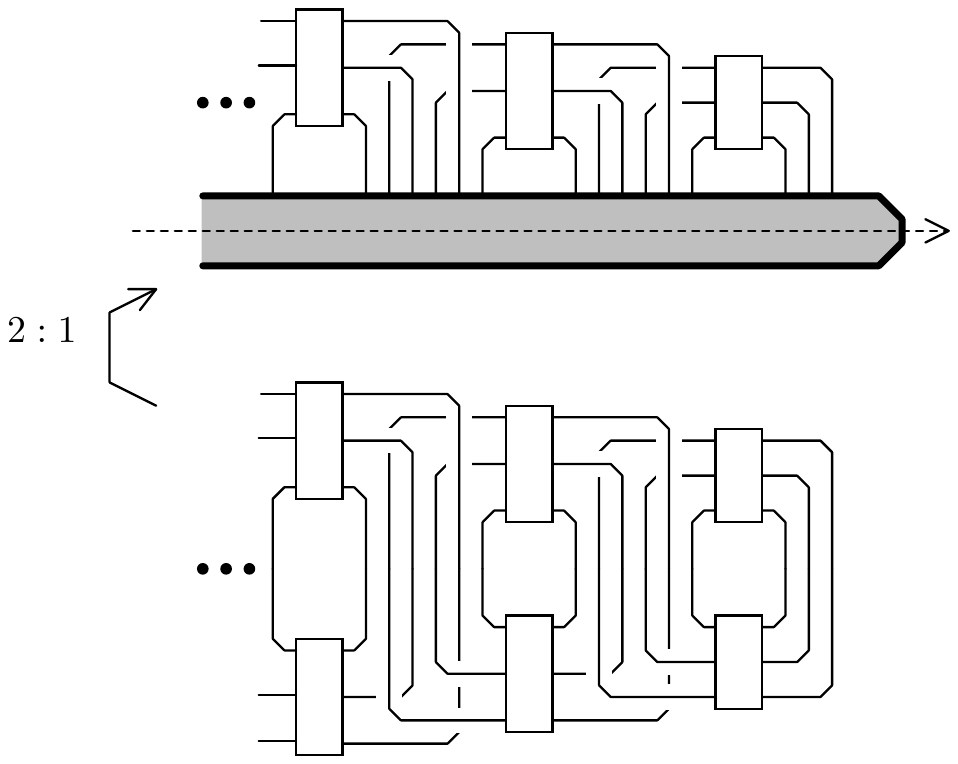}
\caption{Kirby Diagram for double branched covering along the canonical Seifert surface $\Sigma$.}
\label{fig:Kirbydiagram}
\end{center}
\end{figure}

Now one immediately see that these two diagrams are the same.

\end{example}

\end{document}